\documentclass[12pt,a4paper]{amsart}

\usepackage{t1enc}
\usepackage[latin1]{inputenc}
\usepackage[english]{babel}
\usepackage{hyperref}
\usepackage{graphicx}

\usepackage{latexsym}
\usepackage{mathtools}
\usepackage{amsmath,amsfonts,amssymb,amsthm}

\begin{document}

\title{If $L(\chi,1)=0$ then $\zeta(1/2+it)\neq0$.}

\author{Sergio Venturini}
\address{
	S. Venturini:
	Dipartimento Di Matematica,
	Universit\`{a} di Bologna,
	\,\,Piazza di Porta S. Donato 5 ---I-40127 Bologna,
	Italy}
\email{sergio.venturini@unibo.it}


\maketitle

\def\R{{\rm I\kern-.185em R}}
\def\RR{\mathbb{R}}
\def\C{{\rm\kern.37em\vrule height1.4ex width.05em depth-.011em\kern-.37em C}}
\def\CC{\mathbb{C}}
\def\N{{\rm I\kern-.185em N}}
\def\NN{\mathbb{N}}
\def\Z{{\bf Z}}
\def\ZZ{\mathbb{Z}}
\def\Q{\mathbb{Q}}
\def\P{{\rm I\kern-.185em P}}
\def\H{{\rm I\kern-.185em H}}
\def\Aleph{\aleph_0}
\def\ALEPH#1{\aleph_{#1}}
\def\sset{\subset}\def\ssset{\sset\sset}
\def\bar#1{\overline{#1}}
\def\dim{\mathop{\rm dim}\nolimits}
\def\half{\textstyle{1\over2}}
\def\Half{\displaystyle{1\over2}}
\def\mlog{\mathop{\half\log}\nolimits}
\def\Mlog{\mathop{\Half\log}\nolimits}
\def\Det{\mathop{\rm Det}\nolimits}
\def\Hol{\mathop{\rm Hol}\nolimits}
\def\Aut{\mathop{\rm Aut}\nolimits}
\def\Re{\mathop{\rm Re}\nolimits}
\def\Im{\mathop{\rm Im}\nolimits}
\def\Ker{\mathop{\rm Ker}\nolimits}
\def\Fix{\mathop{\rm Fix}\nolimits}
\def\Exp{\mathop{\rm Exp}\nolimits}
\def\sp{\mathop{\rm sp}\nolimits}
\def\id{\mathop{\rm id}\nolimits}
\def\Rank{\mathop{\rm rk}\nolimits}
\def\Trace{\mathop{\rm Tr}\nolimits}
\def\Res{\mathop{\rm Res}\limits}
\def\cancel#1#2{\ooalign{$\hfil#1/\hfil$\crcr$#1#2$}}
\def\prevoid{\mathrel{\scriptstyle\bigcirc}}
\def\void{\mathord{\mathpalette\cancel{\mathrel{\scriptstyle\bigcirc}}}}
\def\n{{}|{}\!{}|{}\!{}|{}}
\def\abs#1{\left|#1\right|}
\def\norm#1{\left|\!\left|#1\right|\!\right|}
\def\nnorm#1{\left|\!\left|\!\left|#1\right|\!\right|\!\right|}
\def\Norm#1{\Bigl|\!\Bigl|#1\Bigr|\!\Bigr|}
\def\upperint{\int^{{\displaystyle{}^*}}}
\def\lowerint{\int_{{\displaystyle{}_*}}}
\def\Upperint#1#2{\int_{#1}^{{\displaystyle{}^*}#2}}
\def\Lowerint#1#2{\int_{{\displaystyle{}_*}#1}^{#2}}
\def\rem #1::#2\par{\medbreak\noindent{\bf #1}\ #2\medbreak}
\def\proclaim #1::#2\par{\removelastskip\medskip\goodbreak{\bf#1:}
\ {\sl#2}\medskip\goodbreak}
\def\ass#1{{\rm(\rmnum#1)}}
\def\assertion #1:{\Acapo\llap{$(\rmnum#1)$}$\,$}
\def\Assertion #1:{\Acapo\llap{(#1)$\,$}}
\def\acapo{\hfill\break\noindent}
\def\Acapo{\hfill\break\indent}
\def\prova{\removelastskip\par\medskip\goodbreak\noindent{\it Dimostrazione.\/\ }}
\def\qed{{$\Box$}\par\smallskip}
\def\BeginItalic#1{\removelastskip\par\medskip\goodbreak
\noindent{\it #1.\/\ }}
\def\iff{if, and only if,\ }
\def\sse{se, e solo se,\ }
\def\rmnum#1{\romannumeral#1{}}
\def\Rmnum#1{\uppercase\expandafter{\romannumeral#1}{}}
\def\smallfrac#1/#2{\leavevmode\kern.1em
\raise.5ex\hbox{\the\scriptfont0 #1}\kern-.1em
/\kern-.15em\lower.25ex\hbox{\the\scriptfont0 #2}}
\def\Left#1{\left#1\left.}
\def\Right#1{\right.^{\llap{\sevenrm
\phantom{*}}}_{\llap{\sevenrm\phantom{*}}}\right#1}
\def\newpi{{\pi\mskip -7.8 mu \pi}} 
\def\dimens{3em}
\def\symb[#1]{\noindent\rlap{[#1]}\hbox to \dimens{}\hangindent=\dimens}
\def\references{\bigskip\noindent{\bf References.}\bigskip}
\def\art #1 : #2 ; #3 ; #4 ; #5 ; #6. \par{#1, 
{\sl#2}, #3, {\bf#4}, (#5), #6.\par\smallskip}
\def\book #1 : #2 ; #3 ; #4. \par{#1, {\bf#2}, #3, #4.\par\smallskip}
\def\freeart #1 : #2 ; #3. \par{#1, {\sl#2}, #3.\par\smallskip}
\def\name{\hbox{Sergio Venturini}}
\def\snsaddress{\indent
\vbox{\bigskip\bigskip\bigskip
\name
\hbox{Scuola Normale Superiore}
\hbox{Piazza dei Cavalieri, 7}
\hbox{56126 Pisa (ITALY)}
\hbox{FAX 050/563513}}}
\def\cassinoaddress{\indent
\vbox{\bigskip\bigskip\bigskip
\name
\hbox{Universit\`a di Cassino}
\hbox{via Zamosch 43}
\hbox{03043 Cassino (FR)}
\hbox{ITALY}}}
\def\bolognaaddress{\indent
\vbox{\bigskip\bigskip\bigskip
\name
\hbox{Dipartimento di Matematica}
\hbox{Universit\`a di Bologna}
\hbox{Piazza di Porta S. Donato 5}
\hbox{40127 Bologna (BO)}
\hbox{ITALY}
\hbox{sergio.venturini@unibo.it}
}}
\def\homeaddress{\indent
\vbox{\bigskip\bigskip\bigskip
\name
\hbox{via Garibaldi, 7}
\hbox{56124 Pisa (ITALY)}}}
\def\doubleaddress{
\vbox{
\hbox{\name}
\hbox{Universit\`a di Cassino}
\hbox{via Zamosch 43}
\hbox{03043 Cassino (FR)}
\hbox{ITALY}
\smallskip
\hbox{and}
\smallskip
\hbox{Scuola Normale Superiore}
\hbox{Piazza dei Cavalieri, 7}
\hbox{56126 Pisa (ITALY)}
\hbox{FAX 050/563513}}}
\def\sergio{{\rm\bigskip
\centerline{Sergio Venturini}
\leftline{\bolognaaddress}
\bigskip}}
\def\a{\alpha}
\def\bg{\beta}
\def\g{\gamma}
\def\G{\Gamma}
\def\dg{\delta}
\def\D{\Delta}
\def\e{\varepsilon}
\def\eps{\epsilon}
\def\z{\zeta}
\def\th{\theta}
\def\T{\Theta}
\def\k{\kappa}
\def\lg{\lambda}
\def\Lg{\Lambda}
\def\m{\mu}
\def\n{\nu}
\def\r{\rho}
\def\s{\sigma}
\def\Sg{\Sigma}
\def\ph{\varphi}
\def\Ph{\Phi}
\def\x{\xi}
\def\om{\omega}
\def\Om{\Omega}

\newtheorem{theorem}{Theorem}[section]
\newtheorem{proposition}[theorem]{Proposition}
\newtheorem{lemma}[theorem]{Lemma}
\newtheorem{corollary}[theorem]{Corollary}

\newtheorem{teorema}{Teorema}[section]
\newtheorem{proposizione}[teorema]{Proposizione}
\newtheorem{corollario}[teorema]{Corollario}

\newtheorem{definition}[theorem]{Definition}

\newtheorem{definizione}[teorema]{Definizione}

\newtheorem{remark}[theorem]{Remark}

\newtheorem{osservazione}[teorema]{Osservazione}
\newtheorem{esempio}[teorema]{Esempio}
\newtheorem{esercizio}[teorema]{Esercizio}
\newtheorem{congettura}[teorema]{Congettura}

\def\definedby{\mathrel{\mathop:}=}
\def\Domain{D}
\def\HoloF{F}
\def\RightHalfSpace#1{H_{#1}^+}
\def\amFunc{u}
\def\cmFunc{u}
\def\cmDomain{I}
\def\tzero{t}
\def\APrev#1{{#1}_0}
\def\AFirst#1{{#1}_1}
\def\ASecond#1{{#1}_2}
\def\ALarge{a}

\def\SCoeff{a}
\def\ECoeff{A}
\def\SignCos{+}

\def\amSet{E}
\def\cmE{E}
\def\invZ{L}

\def\beur{\nu}
\def\Primes{P}
\def\PrimeExponent{\mu}

\def\NumField{K}
\def\NumInt{\mathfrak{o}}
\def\NumIdeal{I}
\def\PrimeIdeal{\mathfrak{p}}

\def\Conj[#1]{{#1}^*}

\def\Ping{A}
\def\Pong{B}
\def\ping{a}
\def\pong{b}

\nocite{article:BatemanOnInghamNoZeroes}
\nocite{article:NarasimhanRemaqueSurZeta}
\nocite{article:InghamNoteOnRiemannZeta}
\nocite{book:InghamTheDistributionOfPrimeNumbers}
\nocite{article:ShapiroOnNonVanishingOfElle}
\nocite{article:NewmanNaturalProofElleNotZero}
\nocite{article:WintnerFundLemmaDirichlet}

\begin{abstract}
Let
\begin{equation*}
	\invZ(s)=\sum_{n=1}^{+\infty}\dfrac{a(n)}{n^s}
\end{equation*}
be a Dirichlet series
were $a(n)$ is a bounded completely multiplicative function.

We prove that if $\invZ(s)$ extends to
a holomorphic function on the open half space
$\Re s >1-\delta$, $\delta>0$
and $\invZ(1)=0$ then such a half space 
is a zero free region of the Riemann zeta function $\zeta(s)$.

Similar results is proven for completely multiplicative functions
defined on the space of the ideals of the ring of the algebraic integers
of a number field of finite degree.
\end{abstract}

\section{\label{section:Intro}Introduction}
The prime numbers theorem and the Dirichlet theorem on primes in arithmetic progressions
(and their generalizations to number fields) rely
respectively on the fact that
\begin{equation}\label{stm::ZetaNotZero}
	\zeta(1+it)\neq0
\end{equation}
for each real $t\neq0$
and that
\begin{equation}\label{stm::ElleNotZero}
	L(\chi,1)\neq0
\end{equation}
for each non principal character $\chi:\NN^+\to\CC$.

Further results on the density of the primes on arithmetic progressions
follows from the inequality
\begin{equation}\label{stm::ElleNotZeroBis}
L(\chi,1+it)\neq0
\end{equation}
for each real $t$.

In the above statements $\zeta(s)$ 
is the famous Riemann zeta function,
which is meromorphic on the whole complex plane and coincides with
the Dirichlet series
\begin{equation*}
	\sum_{n=1}^{+\infty}\dfrac{1}{n^s}
\end{equation*}
when $\Re s>1$,
while $L(\chi,s)$ is
the \emph{Dirichlet $L-$function} associated to the character $\chi$,
which (when $\chi$ is not principal) is
a holomorphic entire function which when $\Re s>1$ coincides with the series
\begin{equation*}
	\sum_{n=1}^{+\infty}\dfrac{\chi(n)}{n^s}.
\end{equation*}

Let us recall that the Riemann zeta function 
is meromorphic on the whole complex plane,
has a unique simple pole at $s=1$ with residue $1$.

The zeroes of $\zeta(s)$ are the so called \emph{trivial zeroes}
\begin{equation*}
	\zeta(-2m)=0,\ m=1,2,\ldots
\end{equation*}
and the remaining ones $\rho$ satisfies
\begin{equation*}
0<\Re\rho<1.
\end{equation*}
The famous \emph{Riemann hypothesis} is that all of them are on the \emph{critical line}
\begin{equation*}
\Re s=\dfrac{1}{2}.
\end{equation*}
Nevertheless it is known that infinite of them are on the critical line.

Here, and in the rest of the paper, $\NN^+$ is the set of the positive integers
and if
\begin{equation*}
a:\NN^+\to\CC
\end{equation*}
is an arbitrary bounded arithmetic function then
\begin{equation*}
	\invZ(a, s)=\sum_{n=1}^{+\infty}\dfrac{a(n)}{n^s}
\end{equation*}
is the corresponding associated Dirichlet series.

We recall that a \emph{Dirichlet character} is an arithmetic function,
that is a function
\begin{equation*}
	\chi:\NN^+\to\CC
\end{equation*}
which is \emph{completely multiplicative}, that is
$\chi(1)=1$ and
\begin{equation*}
\chi(mn)=\chi(m)\chi(n)
\end{equation*}
for each pair of $m,n\in\NN^+$
and which also is periodic of period $q>1$, that is
$\chi(n+q)=\chi(n)$ for each $n\in\NN^+$ 
and $\chi(k)=0$ if $k$ and $q$ are not relatively prime.

Clearly any Dirichlet character $\chi$ is a bounded function
and it easy to show that values of a bounded completely multiplicative functions
are complex numbers $z$ which satisfies $\abs{z}\leq1$.

There are several methods in the literature to achieve \eqref{stm::ElleNotZeroBis}.

One of them is an easy consequence of the following
remarkable result of Ingham \cite{article:InghamNoteOnRiemannZeta}.

\begin{theorem}\label{stm::InghamBase}
Let
\begin{equation*}
	\invZ(a, s)=\sum_{n=1}^{+\infty}\dfrac{a(n)}{n^s}
\end{equation*}
be a Dirichlet series 
where $a:\NN^+\to\CC$ is an arbitrary
bounded completely multiplicative arithmetic function.

Assume that $\invZ(a, s)$ extends to a holomorphic function
on the open half space
\begin{equation*}
	\Re s>\dfrac{1}{2}-\delta
\end{equation*}
with $\delta>0$.
Then
\begin{equation*}
	\invZ(a, 1+it)\neq0
\end{equation*}
for each $t\in\RR$.

\end{theorem}

When 
$a=\chi$, a non principal Dirichlet character, one obtain \eqref{stm::ElleNotZeroBis}.

The Ingham proof of theorem \ref{stm::InghamBase} is quite involved,
but a very simple proof is given by Batemen in \cite{article:BatemanOnInghamNoZeroes}.

The first result of this paper is the following.

\begin{theorem}\label{stm::ZetaRiemannZeroFreeRegion}
Let $\APrev{\sigma}<1$ be given.

If there exists a completely multiplicative bounded arithmetic function
\begin{equation*}
	a:\NN^+\to\CC
\end{equation*}
such that the Dirichlet series
\begin{equation*}
	\invZ(a,s)=\sum_{n=1}^{\infty}\dfrac{a(n)}{n^s}
\end{equation*}
extends holomorphically on the
open half space
\begin{equation*}
	\Re s>\APrev{\sigma}
\end{equation*}
and
\begin{equation*}
	\invZ(a,1)=0
\end{equation*}
then
\begin{equation*}
	\zeta(s)\neq0
\end{equation*}
for each $s$ satisfying $\Re s>\APrev{\sigma}$.

\end{theorem}

The theorem above,
together with the fact that the Riemann zeta function $\zeta(s)$
has (infinite) zeroes on the line $\Re s=1/2$,
easily implies the Ingham result;
see at the end of section \ref{section:MainA} for details.

When $a=\chi$, where $\chi$ is  a non principal Dirichlet character it is easy to see
that the corresponding Dirichlet series $L(\chi, s)$ extends holomorphically
on the half space $\Re s>0$
(actually $L(\chi, s)$ extends holomorphically as an entire function).
If 
\begin{equation*}
L(\chi,1)=0
\end{equation*}
for some non principal Dirichlet character $\chi$
then Theorem \ref{stm::ZetaRiemannZeroFreeRegion}
implies that the Riemann zeta function $\zeta(z)$
wouldn't have any zero on the half space $\Re s>0$,
which is clearly absurd.
The observation above explains the title of this paper.

The author is not able to give any example of a Dirichlet series
$\invZ(a,s)$
as in Theorem \ref{stm::ZetaRiemannZeroFreeRegion}
such that satisfies $\invZ(a,1)=0$ and
is holomorphic
on a half space
\begin{equation*}
	\Re s>\sigma_0
\end{equation*}
with
\begin{equation*}
	\dfrac{1}{2}\leq\sigma_0<1.
\end{equation*}
Indeed the existence of such a series for $\sigma_0=1/2$,
combined with our Theorem \ref{stm::ZetaRiemannZeroFreeRegion},
would imply the Riemann hypothesis,
and as far I know the existence of an
half space $\Re s >\sigma_0$ with $1/2<\sigma_0<1$
which is a zero free region for the Riemann zeta function is
also an open conjecture.

Nevertheless we observe that the converse of
Theorem \ref{stm::ZetaRiemannZeroFreeRegion} also holds.

Indeed if $\lambda(n)$ denotes the \emph{Liouville function}
(see next section for details)
then the meromorphic function
\begin{equation*}
	\invZ(\lambda,s)
	=\sum _{n=1}^{\infty }\dfrac{\lambda(n)}{n^s}
	=\dfrac{\zeta (2s)}{\zeta (s)}
\end{equation*}
satisfies
$
	\invZ(\lambda,1)=0
$ 
and obviously
$\invZ(\lambda,s)$ is holomorphic on the half space
$\Re s>\sigma, 1/2\leq\sigma<1$
if, and only if,
such a half space is a zero free region for $\zeta(s)$.

The proof of Theorem \ref{stm::ZetaRiemannZeroFreeRegion}
is obtained as an elementary consequences of
a general non vanishing principle
for holomorphic functions which are analytic continuations
of exponentials of completely monotone functions
which we think of independent interest.

Let us recall that
a $C^{\infty}$ function
\begin{equation*}
	\cmFunc:\cmDomain\to\RR
\end{equation*}
where $\cmDomain$ is a interval of $\RR$ is
\emph{completely monotone} if for $k=0,1,\ldots$
\begin{equation*}
	(-1)^k\cmFunc^{(k)}(x)\geq0
\end{equation*}
for each $x\in\cmDomain$.

Then our result is the following.

\begin{theorem}\label{stm::cmHoloExtension}
Let $\AFirst{\sigma},\ASecond{\sigma}\in\RR$ with
$\AFirst{\sigma}<\ASecond{\sigma}$
and let
$
f(s)\ 
$
be a holomorphic function defined on the open half space
\begin{equation*}
\Re s >\ASecond{\sigma}.
\end{equation*}

Assume that the restriction of $f(s)$ to the half line
$
]\ASecond{\sigma},+\infty[
$
is a real completely monotone function
and
\begin{equation*}
F(s)\definedby\exp f(s)
\end{equation*}
extends holomorphically on the open half space
\begin{equation*}
	\Re s >\AFirst{\sigma},
\end{equation*}

Then the function $f(s)$ also extends holomorphically
on the open half space
\begin{equation*}
	\Re s >\AFirst{\sigma}
\end{equation*}
and hence
\begin{equation*}
	F(s)=\exp f(s)\neq0
\end{equation*}
when $\Re s>\AFirst{\sigma}$.
\end{theorem}

It is quite surprising that our approach
makes it unnecessary to use anywhere
the standard Euler product expansion of such Dirichlet series.

Actually the Euler product expansion is necessary to prove the non vanishing
in the half space of absolute convergence of the Dirichlet series associate to
\emph{multiplicative} but not completely multiplicative functions, that is to
functions
\begin{equation*}
	a:\NN^+\to\CC
\end{equation*}
such that
\begin{equation*}
	a(mn)=a(m)a(n)
\end{equation*}
when $m$ and $n$ are relatively prime.
But in this paper we do not consider multiplicative arithmetic function.
For \emph{completely multiplicative} arithmetic functions
we obtain directly their representation as exponential of Dirichlet series
without using its product expansion;
see (the proofs of)
Lemma \ref{stm::DiriInftyLim}, and 
Proposition \ref{stm::BeurExp} for details.

Let us now describe the content of the paper.

In secion \ref{section:Prereq} we recall
basic facts on arithmetic functions and 
the associated Dirichlet series that
we need in the rest of the paper.

The proof of
Theorem \ref{stm::cmHoloExtension}
and
Theorem \ref{stm::InghamBase}
are given respectively in
section \ref{section:Pringsheim} and \ref{section:MainA}.

In section \ref{section:PringsheimEx}
we give a refined version of theorem 
\ref{stm::cmHoloExtension} which is used
in section \ref{section:MainEx}
to prove theorem \ref{stm::ZetaBeurZeroFreeRegion},
the main result of this paper,
which extends to a class of generalized Dirichlet series
the above theorem \ref{stm::InghamBase},
including,
among the others,
the Dirichlet series associated to
completely multiplicative functions
defined on the ideals of the
ring of the integers of
a number fields.

Non vanishing theorem 
for general $L-$type functions on the boundary of
the half space of absolute convergence
are then obtained
in section \ref{section:Corollaries}.

We end this introduction
with a ``toy application'' of Theorem \ref{stm::cmHoloExtension}
giving three proofs of
$ 
\zeta(1+it)\neq0.
$ 
Each of them contains  ``themes'' which will be expanded in the rest of the paper.

All of them start observing that when $\Re s>1$ we have
\begin{equation*}
	\zeta(s)=\exp\left(\sum_{n=1}^{\infty}\dfrac{\Lambda (n)}{n^s\log n}\right)
\end{equation*}
where $\Lambda(n)$ is the \emph{von Mangolt function} (see, e.g., next section).

Assume that 
\begin{equation*}
	\zeta(1+i\tzero)=0
\end{equation*}
for some $\tzero>0$.
Following \cite[pag. 199]{article:OggOnSatoTateConjecture}
consider the function
\begin{equation*}
	F(s)=\zeta(s)^2\zeta(s+i\tzero)\zeta(s-i\tzero).
\end{equation*}
Then $F(s)$ has removable singularities at $s=1$, $s=1+i\tzero$ and $s=1-i\tzero$
and hence is a holomorphic entire function.

When $\Re s>1$ we have then
\begin{equation*}
	F(s) 
	=\exp f(s).
\end{equation*}
where
\begin{equation*}
	f(s)
	=\sum_{n=1}^{\infty}\dfrac{2\bigl(1+\Re(n^{-i\tzero})\bigr)\Lambda (n)}{n^s\log n}.
\end{equation*}
Since $\Re n^{-i\tzero}=\cos\bigl(\tzero\log(n)\bigr)\geq-1$
the Dirichlet series $f(s)$
has non negative coefficients
and hence the function 
\begin{equation*}
	]1,+\infty[\ni\sigma\mapsto f(\sigma)
	\in\RR
\end{equation*}
is completely monotone.

Theorem \ref{stm::cmHoloExtension} implies that
$f(s)$ extents to an entire holomorphic function and
$F(s)=\exp f(s)$ is a not vanishing entire holomorphic function;
a classical Landau's theorem
(see theorem \ref{stm::LandauStandard})
implies that
the series $f(s)$ then converges for all $s\in\CC$.

Now we have three ways to conclude the proof.

The first one begins by
observing that at $s=1+i\tzero$ the factor
$\zeta(s)^2$ of the function $F(s)$ has a zero of the second order
while the factor $\zeta(s-i\tzero)$ has a simple pole.

Since $F(1+i\tzero)\neq0$ 
then the remaining factor $\zeta(s+i\tzero)$
must necessarily have a simple pole at $s=1+i\tzero$, that is
the Riemann zeta function $\zeta(s)$ would have
also an other pole at $s=1+2i\tzero$,
which is absurd.

The second one 
follows from the fact that if $F(s)$ never vanishes then
$\zeta(s)$ also  never vanishes when $s\neq1$.
In particular it follows that
$\zeta(s)$ does not vanishes
neither when $s=-2m$, $m=1,2,\ldots$
nor when $\Re s=1/2$
and this is not the case $\ldots$

For the third
let denote by $\Primes=\{2,3,5,\ldots\}$ the set of the positive prime numbers and
set $a(n)=n^{-i\tzero}$.
If $\sigma\in\RR$ then
\begin{eqnarray*}
	f(\sigma)&=&\sum_{n=1}^{\infty}
	\dfrac{2\bigl(1+\Re a(n)\bigr)\Lambda(n)}{n^s\log(n)}
	=\sum_{p\in\Primes}\sum_{m=1}^\infty\dfrac{2\bigl(1+\Re a(p)^m\bigr)}{m p^{m\sigma}}\\
	&\geq&\sum_{p\in\Primes}\sum_{m=1}^2\dfrac{2\bigl(1+\Re a(p)^m\bigr)}{m p^{m\sigma}}\\
	&\geq&\sum_{p\in\Primes}\dfrac{\bigl(2+\Re a(p)+\Re a(p)^2\bigr)}{p^{2\sigma}}.
\end{eqnarray*}

Observing that
\begin{equation*}
\boxed{
\abs{w}\leq1\implies\Re w+\Re w^2\geq-\dfrac{9}{8}
}
\end{equation*}
with equality at
\begin{equation*}
w=\dfrac{1}{4}\pm i\dfrac{\sqrt{15}}{4}
\end{equation*}
we then obtain
\begin{equation*}
f(\sigma)\geq\dfrac{7}{8}\sum_{p\in\Primes}\dfrac{1}{p^{2\sigma}}.
\end{equation*}
Since the series of the reciprocal of the prime numbers diverges
then the series $f(s)$
also diverges at $s=1/2$.

For a fourth (easy) proof see corollary \ref{stm::ElleZetaPNT}.

\section{\label{section:Prereq}Prerequisites}

In this paper we need only the basic results on
(multiplicative) arithmetic function and
their associated
Dirichlet series.
Basic references include the first chapters of
\cite{book:ApostolANT},
\cite{book:NarkiewiczANT3},
\cite{book:NeukirchANT},
\cite{book:MontgomeryVaughan2006},
\cite{book:OverholtCourseAnalyticNT} and
\cite{book:TenenbaumAnalyticNumberTheory}.

Here we review basic material that we need in this paper.

The already mentioned \emph{von Mangoldt function} is the arithmetic function
\begin{equation*}
\Lambda (n)=
\begin{cases}
\log p&\text{if }n=p^k\text{ for some prime }p\text{ and integer }k\geq 1,\\
0     &\text{otherwise.}
\end{cases}
\end{equation*}

It satisfies the identity
\begin{equation}\label{baseMangolt}
	\sum_{d|n}\Lambda(d)=\log n.
\end{equation}

The \emph{Liouville function} is the completely multiplicative functions
\begin{equation*}
\lambda (n)=(-1)^{\Omega (n)}.
\end{equation*}
where $\Omega(n)$ is the number of prime factors of $n$, counted with multiplicity.

The Liouville function and the Riemann zeta function are related by the identity
\begin{equation*}
	\dfrac{\zeta (2s)}{\zeta (s)}=\sum _{n=1}^{\infty }\dfrac{\lambda(n)}{n^s}
\end{equation*}

We will need the following elementary result on Dirichlet series.

\begin{lemma}\label{stm::beur::FirstCoeffBase}
Let
\begin{equation*}
F(s)=\sum_{n=1}^{\infty}\dfrac{a(n)}{n^s}.
\end{equation*}
where $a:\NN^+\to\CC$
is a bounded arithmetic function.
Then, 
\begin{equation*}
\lim_{\RR\ni\sigma\to\infty}F(\sigma)=a(1).
\end{equation*}
\end{lemma}

The following statement is a classical result of Landau 
(see e.g. \cite[Theorem 1.7, pag. 16]{book:MontgomeryVaughan2006}
or \cite[Lemma 1, pag. 314]{book:LangAlgebraicNumerTheory})

\begin{theorem}\label{stm::LandauStandard}
Let $\APrev{\sigma},\AFirst{\sigma}\in\RR$ with
$\APrev{\sigma}<\AFirst{\sigma}$
and let
\begin{equation*}
	f(s)=\sum_{n=1}^{\infty}\dfrac{a_n}{n^s}
\end{equation*}
be a Dirichlet series with non negative coefficients $a_n\geq0$.

Assume that $f(s)$ converges when $\Re s>\AFirst{\sigma}$
and extends to a holomorphic function on $\Re s>\APrev{\sigma}$.

Then the series $f(s)$ also converges when $\Re s>\APrev{\sigma}$.

\end{theorem}

\section{\label{section:Pringsheim}A non vanishing principle}
In this section we prove
theorem \ref{stm::cmHoloExtension}.

We begin recalling a theorem due to Pringsheim
which appeared in
\cite{article:Pringsheim1894AnCont}
also known as the Pringsheim-Vivanti theorem.
See also
\cite[Theorem 5.7.1]{book:HilleAFTVol1},
and
\cite[Theorem 8.2.2]{book:SansoneGerretsen}.

\begin{theorem}\label{stm::Pringsheim}
Let
\begin{equation*}
	f(z)=\sum_{n=0}^{\infty}a_nz^n
\end{equation*}
be a convergent power series with radius of convergence $R$,
with $0<R<+\infty$.

If $a_n\geq0$ for each $n$ then
it is not possible to extend $f(z)$ holomorphically
in a neighbourhood of $z=R$.

\end{theorem}

The following proposition is the basic technical tool
of the paper.

\begin{proposition}\label{stm::ExpRadiusIsEqualEx}	
Let
\begin{equation*}
	\cmE:\CC\to\CC
\end{equation*}
be a holomorphic entire function such that $\cmE(\RR)\sset\RR$,
$\cmE'(t)>0$ for each $t>0$ and
\begin{equation*}
	\lim_{t\to+\infty}\cmE(t)=+\infty.
\end{equation*}
Let
\begin{equation*}
f(z)=\sum_{n=1}^{+\infty}\SCoeff_nz^n
\end{equation*}
and
\begin{equation*}
F(z)=\sum_{n=1}^{+\infty}\ECoeff_nz^n
\end{equation*}
be two convergent powers series
with real coefficients
such that
\begin{equation*}
F(z)=\cmE\bigl(f(z)\bigr)
\end{equation*}
in a neighbourhood of $z=0$.

If the coefficients $\SCoeff_n$ of $f(z)$ are
not negative 
then
the series $f(z)$ and $F(z)$ 
have the same radius of convergence.

\end{proposition}

\begin{proof}
If the function $f$ is constant then $F$ also is constant and
in this case the assertion is obvious.

We hence assume that $f$ is not constant.

Let $r$ and $R$ denote the radius of convergence respectively of
the series $f(z)$ and $F(z)$.

Then $F(z)=\cmE\bigl(f(z)\bigr)$ is
holomorphic on the disc $\abs{z}<r$.
Standard theorems of one complex variable imply that
$R\geq r$.

We now assume that $R>r$ and derive a contradiction.

Since $f$ is not constant $a_{n_0}>0$ for some $n_0>0$
and hence
\begin{eqnarray*}
	f(x)&=&\sum_{n=0}^{\infty}na_nx^{n}
	\geq a_{n_0}x^{n_0}>0,\\
	f'(x)&=&\sum_{n=1}^{\infty}na_nx^{n-1}
	\geq n_0a_{n_0}x^{n_0-1}>0
\end{eqnarray*}
when $0<x<r$.

Then if $0<x<r$ we have
\begin{equation*}
	F'(x)=\cmE'\bigl(f(x)\bigr)f'(x)>0
\end{equation*}
and hence,
by the Lagrange theorem,
the function $F(x)$ is strictly increasing on
the closed interval $[0,r]$.
Since $F(r)>F(0)$ then
\begin{equation*}
F(x)>F(0),\ 0<x< r'
\end{equation*}
for some $r'$ with
\begin{equation*}
r<r'<R.
\end{equation*}

The hypotheses on $\cmE$ implies that
the inverse function
\begin{equation*}
\cmE^{-1}:]\cmE(0),+\infty[\to]0,+\infty]
\end{equation*}
is well defined and real analytic.

Being
\begin{equation*}
	F(x)>F(0)=\cmE\bigl(f(0)\bigr)\geq\cmE(0)
\end{equation*}
when $0<x<r'$,
it follows that the function
\begin{equation*}
]0,r'[\ni x\mapsto u(x)\definedby\cmE^{-1}\bigl(F(x)\bigr)\in\RR
\end{equation*}
is a well defined real analytic function
which coincides with $f(x)$ when $0<x<r$.

The power expansion of $u(x)$ at $x=r$
then defines a holomorphic extension of the function $f(z)$
in a neighbourhood of $z=r$
and this contradicts theorem \ref{stm::Pringsheim}.

\end{proof}

We are now ready to prove theorem \ref{stm::cmHoloExtension}.

So, let
$\AFirst{\sigma},\ \ASecond{\sigma}\in\RR$ and
$f(s)$, $F(s)$ be given as in theorem \ref{stm::cmHoloExtension}.
Let us fix $\ALarge>\ASecond{\sigma}$.
Then the functions
\begin{equation*}
f_\ALarge(z)=f(a-z)
\end{equation*}
and
\begin{equation*}
F_\ALarge(z)=F(a-z).
\end{equation*}
are holomorphic respectively on the disks
$\abs{z}<\ALarge-\ASecond{\sigma}$ and 
$\abs{z}<\ALarge-\AFirst{\sigma}$.

The relation
$F_\ALarge(z)=\exp\bigl(f_\ALarge(z)\bigr)$
holds on the smaller disk $\abs{z}<\ALarge-\ASecond{\sigma}$
and for each integer $n\geq0$
\begin{equation*}
	f_\ALarge^{(n)}(0)=(-1)^n f^{(n)}(a)\geq0.
\end{equation*}
Theorem \ref{stm::ExpRadiusIsEqualEx}
(with $\cmE=\exp$)
implies then that the
radius of convergence of
the power series expansions of $f_\ALarge(z)$ and $F_\ALarge(z)$
at $z=0$ are the same.

Since $F_\ALarge(z)$ is holomorphic on the disk
$\abs{z}<\ALarge-\AFirst{\sigma}$
such a common radius of convergence is at least $\ALarge-\AFirst{\sigma}$.

Hence the function $f_\ALarge(z)$,
defined in the smaller disk 
$\abs{z}<\ALarge-\ASecond{\sigma}$,
extends holomorphically on the bigger disk
$\abs{z}<\ALarge-\AFirst{\sigma}$.

The formula
\begin{equation*}
	f(s)=f_\ALarge(a-s)
\end{equation*}
then defines a holomorphic extension of the function $f(s)$
on the disk $\abs{z}<\ALarge-\AFirst{\sigma}$.

The observation that the union of all the disks
\begin{equation*}
\abs{s-\ALarge}<\ALarge-\AFirst{\sigma},\ \ALarge>\ASecond{\sigma}
\end{equation*}
is the open half space
\begin{equation*}
\Re s>\AFirst{\sigma}
\end{equation*}
completes the proof.

\section{\label{section:MainA}Proof of theorem \ref{stm::ZetaRiemannZeroFreeRegion}}

Here the proof of theorem \ref{stm::ZetaRiemannZeroFreeRegion}.

So, let $\APrev{\sigma}<1$ and
let $a:\NN^+\to\CC$ be a bounded completely multiplicative functions.
Assume that the Dirichlet series
\begin{equation*}
	\invZ(a, s)=\sum_{n=1}^{+\infty}\dfrac{a(n)}{n^s}
\end{equation*}
extends holomorphically in the half space $\Re s>\APrev{\sigma}$
and
\begin{equation*}
	\invZ(a,1)=0.
\end{equation*}

We denote the function $L(a,s)$ simply with $L(s)$.

Since $a(n)$ is completely multiplicative and bounded then necessarily
\begin{equation*}
	\abs{a(n)}\leq1
\end{equation*}
for each $n\in\NN^+$
and hence the series $L(s)$
converges absolutely when $\Re s>1$.

The following lemma is well known,
but we include a very elementary proof.
\begin{lemma}\label{stm::DiriInftyLim}
The series $\invZ(s)$ satisfies
\begin{equation*}
\invZ(s)=\exp\left(
\sum_{n=2}^{\infty}\dfrac{a(n)\Lambda(n)}{n^s\log n}
\right)
\end{equation*}
when $\Re s>1$
\end{lemma}

\begin{proof}
Let
\begin{equation*}
f(s) =\sum_{n=2}^{\infty}\dfrac{a(n)\Lambda(n)}{n^s\log n}.
\end{equation*}
Then
\begin{equation*}
f'(s) =-\sum_{n=2}^{\infty}\dfrac{a(n)\Lambda(n)}{n^s}.
\end{equation*}
Since the arithmetic functions $a(n)$ and $\lambda(n)$ are completely multiplicative
then the identity \eqref{baseMangolt}
easily implies that
\begin{equation*}
f'(s)\invZ(s)=\invZ'(s).
\end{equation*}
Then the derivative of the function
\begin{equation*}
s\mapsto\exp\bigl(-f(s)\bigr)\invZ(s)
\end{equation*}
vanishes and hence is constant, that is
\begin{equation*}
\invZ(s)=c\exp\bigl(f(s)\bigr)
\end{equation*}
for some constant $c$.
Now put $s=\sigma\in\RR$;
and take the limit as $\sigma\to+\infty$;
lemma \ref{stm::beur::FirstCoeffBase} then implies
\begin{equation*}
1=a(1)=c\exp(0)=c.
\end{equation*}

\end{proof}

When $a(n)=1$ then $\invZ(s)=\zeta(s)$ and hence we obtain
\begin{equation*}
	\zeta(s)=\exp\left(
	\sum_{n=1}^{\infty}\dfrac{\Lambda(n)}{n^s\log n}
	\right)
	.
\end{equation*}
Consider the function
\begin{equation*}
	F(s)=\zeta(s)^2\invZ(s)\bar{\invZ(\bar{s})}.
\end{equation*}
Since $\invZ(1)=0$,
then the function $F(s)$ is holomorphic
on the open half  $\Re s >\APrev{\sigma}$
and when $\Re s>1$ also satisfies
\begin{equation*}
	F(s)=\exp f(s)
\end{equation*}
where
\begin{equation*}
f(s)=\sum_{n=1}^{\infty}\dfrac{2\bigl(1+\Re a(n)\bigr)\Lambda(n)}{n^s\log n}.
\end{equation*}
Clearly $f(s)$ is a Dirichlet series (absolutely) convergent when $\Re s>1$
with non negative coefficients
and hence the restriction of $f(s)$ to the real half line $]1,+\infty[$
is a completely monotone function.

Theorem \ref{stm::cmHoloExtension} implies that
$F(s)$ never vanishes 
on the half space $\Re s>\APrev{\sigma}$.
This forces
$\zeta(s)\neq0$
when $\Re s>\APrev{\sigma}$
and the proof of Theorem \ref{stm::ZetaRiemannZeroFreeRegion}
is completed.

We observe that the Ingham Theorem \ref{stm::InghamBase}
easily follows from the theorem \ref{stm::ZetaRiemannZeroFreeRegion}.

Indeed let 
\begin{equation*}
F(s)=\sum_{n=1}^{+\infty}\dfrac{a(n)}{n^s}
\end{equation*}
be as in theorem \ref{stm::InghamBase}
and assume that $F(1+it)=0$ for some $t\in\RR$.

Then the function $\invZ(s)=F(s+it)$ satisfies the hypotheses
of theorem \ref{stm::ZetaRiemannZeroFreeRegion}.
Indeed $\invZ(1)=0$ and when $\Re s>1$ we have
\begin{equation*}
	\invZ(s)=F(s+it)=\sum_{n=1}^{\infty}\dfrac{a(n)n^{-it}}{n^s}.
\end{equation*}
The map $n\mapsto a(n)n^{-it}$
is a bounded completely multiplicative arithmetic function.
Theorem \ref{stm::ZetaRiemannZeroFreeRegion} then implies that
the open half space $\Re s>1/2-\delta$
is a zero free region for the Riemann zeta function
and this is not the case.

\section{\label{section:PringsheimEx}A refined non vanishing principle}
In order to extend theorem \ref{stm::cmHoloExtension}
we need a refined version of theorem \ref{stm::Pringsheim}.

\begin{theorem}\label{stm::RealAnalyticAMPringsheim}
Let $R,\delta>0$ be positive real numbers.
Let
\begin{equation*}
\amFunc:]-\delta,R[\to\RR
\end{equation*}
be a real analytic function.

If
\begin{equation*}
a_n\definedby\dfrac{\amFunc^{(n)}(0)}{n!}\geq0,\ n=0,1,\ldots,
\end{equation*}
then the radius of convergence of the series
\begin{equation*}
\sum_{n=0}^{\infty}a_nz^n
\end{equation*}
is greater than $R$.
\end{theorem}

\begin{proof}
Let denote by $r>0$ the radius of convergence of the series
\begin{equation*}
f(z)=\sum_{n=0}^{\infty}a_nz^n.
\end{equation*}
Then $f(z)$ is holomorphic on the disc $\abs{z}<r$
and $f(x)=u(x)$ when $0\leq x< r$.

Assume that $r<R$.
As $\amFunc(x)$ is real analytic 
it follows that
the power series expansion of $\amFunc(x)$ at $x=r$
defines a holomorphic extension of the function $f(z)$
in a neighbourhood of $z=r$.
But this is not allowed by theorem \ref{stm::Pringsheim}.

It necessarily follows that $r\geq R$
and we are done.

\end{proof}

Let us recall that
a $C^{\infty}$ function
\begin{equation*}
\cmFunc:\cmDomain\to\RR
\end{equation*}
where $\cmDomain$ is a interval of $\RR$ is
\emph{absolutely monotone} if for $k=0,1,\ldots$
\begin{equation*}
\cmFunc^{(k)}(x)\geq0
\end{equation*}
for each $x\in\cmDomain$.

The absolutely monotone functions
were introduced by S. Bernstein in \cite{article:Bernstein1914AM}.
For a comprehensive treatment of the theory of such a functions
(and their cousins the completely monotone ones)
see e.g.,
\cite[Chapter IV]{book:WidderLaplace};
see also
\cite[Chapter 1]{book:SchillingEtcBernsteinFunctions}.

The following is an extension 
theorem \ref{stm::cmHoloExtension}.

\begin{theorem}\label{stm::cmHoloExtensionEx}
Let $\AFirst{\sigma},\ASecond{\sigma}\in\RR$ with
$\AFirst{\sigma}<\ASecond{\sigma}$
and let
$
f(s)\ 
$
be a holomorphic function defined on the open half space
\begin{equation*}
\Re s >\ASecond{\sigma}.
\end{equation*}

Assume that the restriction of $f(s)$ to the half line
$
]\ASecond{\sigma},+\infty[
$
is a real completely monotone function
and
\begin{equation*}
F(s)=\exp f(s)
\end{equation*}
extends holomorphically in an open neighbourhood of the real interval
$
]\AFirst{\sigma},\ASecond{\sigma}].
$

Then both the functions $f(s)$ and $F(s)$ extend holomorphically
on the open half space
\begin{equation*}
\Re s >\AFirst{\sigma},
\end{equation*}
the relation
\begin{equation*}
F(s)=\exp f(s)
\end{equation*}
still holds on such a half space
and hence
\begin{equation*}
F(s)\neq0,\ \Re s>\AFirst{\sigma}.
\end{equation*}
\end{theorem}

\begin{proof}
As in the proof of theorem \ref{stm::cmHoloExtensionEx}
we choose $\ALarge>\ASecond{\sigma}$
and consider the functions
\begin{equation*}
f_\ALarge(z)=f(a-z)
\end{equation*}
and
\begin{equation*}
F_\ALarge(z)=F(a-z).
\end{equation*}
They are both holomorphic functions on the disk
$\abs{z}<\ALarge-\ASecond{\sigma}$ and the function $F_\ALarge(z)$
is also holomorphic in a neighbourhood of the segment $[0, \ALarge-\AFirst{\sigma}[$.

Since $f$ is completely monotone on the interval $]\ASecond{\sigma}, 2\ALarge-\ASecond{\sigma}[$
then $f_\ALarge$ is absolutely monotone on the interval $]\ASecond{\sigma}-\ALarge,\ALarge-\ASecond{\sigma}[$.

Since the composition of absolutely monotone function is absolutely monotone it follows that
the function $F_\ALarge(z)=\exp\bigl(f_\ALarge(z)\bigr)$ is also absolutely monotone on
the interval $]\ASecond{\sigma}-\ALarge,\ALarge-\ASecond{\sigma}[$.

But the function $F_\ALarge(z)$ is real analytic on 
the interval $]\ASecond{\sigma}-\ALarge,\ALarge-\AFirst{\sigma}[$.

Then theorem \ref{stm::RealAnalyticAMPringsheim} implies that
the radius of convergence of the powers expansion of $F_a(t)$ at $t=0$
is greater that $\ALarge-\AFirst{\sigma}$.

Proposition \ref{stm::ExpRadiusIsEqualEx} implies that
both the powers expansions of $f_\ALarge(z)$ and $F_\ALarge(z)$ at $z=0$
have the same radius of convergence at least $\ALarge-\AFirst{\sigma}$
and hence provide holomorphic extensions respectively of $f(s)$ and $F(s)$
on the disc $\abs{s-\ALarge}<\ALarge-\AFirst{\sigma}$.

As in the proof of theorem \ref{stm::cmHoloExtensionEx}
we conclude the proof by observing that
the union of all the disks
\begin{equation*}
\abs{s-\ALarge}<\ALarge-\AFirst{\sigma},\ \ALarge>\ASecond{\sigma}
\end{equation*}
is the open half space
$
\Re s>\AFirst{\sigma}.
$

\end{proof}
\section{\label{section:MainEx}The main theorem}
In this section we consider a fixed  real arithmetic function
\begin{equation*}
	\beur:\NN^+\to\RR
\end{equation*}
satisfying the following properties:
\begin{enumerate}
\item\label{beur::Multiplicative}
$\beur$ is completely multiplicative, that is
\begin{equation*}
	\beur(mn)=\beur(n)\beur(m)
\end{equation*}
for each $m,n\in\NN^+$;
\item\label{beur::Positive}
$
	\beur(n)>1\ 
$
for each $n>1$;
\item\label{beur::Zeta}
For a certain $\AFirst{\sigma}>0$
the (generalized) Dirichlet series
\begin{equation*}
	Z(s)=\sum_{n=1}^{\infty}\dfrac{1}{\beur(n)^s}
\end{equation*}
converges for $\Re s>\AFirst{\sigma}$
and 
\begin{equation*}
	\lim_{\sigma\to\AFirst{\sigma}^+}
	Z(\sigma)=+\infty
\end{equation*}
\end{enumerate}

Observe that condition \eqref{beur::Zeta} forces
\begin{equation}\label{beu::Infty}
	\lim_{n\to+\infty}\beur(n)=+\infty
\end{equation}
and hence:
\begin{proposition}\label{beur::StriclyPositive}
There exists $\beur_0>1$ such that
\begin{equation*}
	\beur(n)\geq\beur_0
\end{equation*}
for each $n>1$.
\end{proposition}

Observe that we do not require the sequence $n\mapsto\beur(n)$
to be not decreasing.

The basic example is given by the Dedekind zeta function of number field
(see, e.g.,
\cite{book:LangAlgebraicNumerTheory},
\cite{book:NeukirchANT}
).
Let $\NumField$ be a number field,
that is a finite extension of the rational field $\Q$
and let $\NumInt$
be the ring of the algebraic integers in $\NumField$.
Consider a one to one bijection
\begin{equation}\label{eq::PrimeBijection}
	p\mapsto\PrimeIdeal_p
\end{equation}
between the positive prime numbers $P=\{2,3,5,\ldots\}$
and the non zero prime (i.e. maximal) ideals of $\NumInt$.

Since the non zero ideals of $\NumInt$ factor uniquely
as product of non zero prime ideals then
we can extend \eqref{eq::PrimeBijection} to a bijection
\begin{equation*}
n\mapsto\NumIdeal_n
\end{equation*}
between the positive integers $n\in\NN^+$
in such a way that if
\begin{equation*}
	n=p_1^{\PrimeExponent_1}\cdots p_k^{\PrimeExponent_k}
\end{equation*}
then
\begin{equation*}
	\NumIdeal_n=
		\PrimeIdeal_{p_1}^{\PrimeExponent_1}
		\cdots\PrimeIdeal_{p_k}^{\PrimeExponent_k}.
\end{equation*}
Then the arithmetic function
\begin{equation*}
	\beur(n)=\mathfrak{N}(\NumIdeal_n)
\end{equation*}
where $\mathfrak{N}(\NumIdeal_n)$ 
is the \emph{norm} of the ideal $\PrimeIdeal_n$
satisfies the properties $(1)$, $(2)$ and $(3)$ above.

Then
\begin{equation*}
	Z(s)=\sum_{n=1}^{\infty}\dfrac{1}{\beur(n)^s}=\zeta_\NumField(s),
\end{equation*}
where
\begin{equation*}
\zeta_\NumField(s)=\sum_{\NumIdeal}\dfrac{1}{\mathfrak{N}(I)^s}
\end{equation*}
is the Dedekind zeta function of the number field $\NumField$.

Other examples arise from abstract arithmetic semigroups
(see \cite{book:KnopfmacherAbstractANT}),
Beurling's generalized prime numbers
(see \cite[Section 8.4, pag. 266]{book:MontgomeryVaughan2006},
\cite{article:BeurlingPrimes}),
and
the zeta function associated to
an arithmetical scheme (see \cite{article:SerreZandL1965}).

The purpose of this section is to extend
theorem \ref{stm::ZetaRiemannZeroFreeRegion}
to (generalized) Dirichlet series of the form
\begin{equation*}
	\sum_{n=1}^{+\infty}\dfrac{a(n)}{\beur(n)^s}.
\end{equation*}
Clearly if the arithmetic function $a(n)$ is
bounded the such a series defines a holomorphic function
on the open half space $\Re s>\AFirst{\sigma}$.

We need to prove some elementary properties of such a series.

Let us begin by extending
Lemma \ref{stm::beur::FirstCoeffBase}
and
Lemma \ref{stm::DiriInftyLim}
to such a series.

\begin{lemma}\label{stm::beur::FirstCoeff}
Let
\begin{equation*}
	F(s)=\sum_{n=1}^{\infty}\dfrac{a(n)}{\beur(n)^s}.
\end{equation*}
where $a:\NN^+\to\CC$
is a bounded arithmetic function.
Then
\begin{equation*}
	\lim_{\sigma\to\infty}F(\sigma)=a(1).
\end{equation*}
\end{lemma}

\begin{proof}
Let $\beur_0>1$ such that $\beur(n)\geq\beur_0$ when $n>1$.

Fix $\ASecond{\sigma}>\AFirst{\sigma}$.
Then for each $\sigma>\ASecond{\sigma}$ we have
\begin{equation*}
	\abs{F(\sigma)-a(1)}
	\leq\sum_{n=2}^{\infty}\dfrac{\abs{a(n)}}{\beur(n)^\sigma}
	\leq\dfrac{1}{\beur_0^{\sigma-\ASecond{\sigma}}}
		\sum_{n=2}^{\infty}\dfrac{\abs{a(n)}}{\beur(n)^{\ASecond{\sigma}}}.
\end{equation*}
Since
\begin{equation*}
	\lim_{\sigma\to+\infty}\dfrac{1}{\beur_0^{\sigma-\ASecond{\sigma}}}=0
\end{equation*}
the assertion follows.

\end{proof}

We define the $\beur-$von Mangold function
\begin{equation*}
\Lambda_\beur (n)={\begin{cases}\log\beur( p)&{\text{if }}n=p^{k}{\text{ for some prime }}p{\text{ and integer }}k\geq 1,\\0&{\text{otherwise.}}\end{cases}}
\end{equation*}

As in the classical case it is easy to prove that for each $n>0$
\begin{equation}\label{beurMangolt}
	\sum_{d|n}\Lambda_\beur(d)=\log\beur(n)
\end{equation}

\begin{proposition}\label{stm::BeurExp}
Let
\begin{equation*}
F(s)=\sum_{n=1}^{\infty}\dfrac{a(n)}{\beur(n)^s}.
\end{equation*}
where $a:\NN^+\to\CC$
is a bounded completely multiplicative arithmetic function.
Then  when $\Re s >\AFirst{\sigma}$
\begin{equation*}
F(s) =\exp\left(
	\sum_{n=2}^{+\infty}\dfrac{a(n)\Lambda_\beur(n)}{\beur(n)^s\log\beur(n)}
\right)
\end{equation*}
\end{proposition}

\begin{proof}
Let
\begin{equation*}
f(s) =\sum_{n=2}^{+\infty}\dfrac{a(n)\Lambda_\beur(n)}{\beur(n)^s\log\beur(n)}.
\end{equation*}
Then
\begin{equation*}
f'(s) =-\sum_{n=2}^{+\infty}\dfrac{a(n)\Lambda_\beur(n)}{\beur(n)^s}.
\end{equation*}
Since the arithmetic functions $a(n)$ and $\lambda(n)$ are completely multiplicative
then \eqref{beurMangolt} easily implies that
\begin{equation*}
	f'(s)F(s)=F'(s).
\end{equation*}
Then the derivative of the function
\begin{equation*}
	\exp\bigl(-f(s)\bigr)F(s)
\end{equation*}
vanishes, that is
\begin{equation*}
	F(s)=c\exp\bigl(f(s)\bigr)
\end{equation*}
for some constant $c$.
Then put $s=\sigma\in\RR$ and take the limit as $\sigma\to+\infty$;
lemma \ref{stm::beur::FirstCoeff} then implies
\begin{equation*}
	1=a(1)=c\exp(0)=c.
\end{equation*}

\end{proof}

Let us denote by $\Primes$ the set of positive prime numbers.
\begin{proposition}\label{stm::beuPrimesDiverges}
\begin{equation*}
	\lim_{\sigma\to\AFirst{\sigma}^+}\sum_{p\in\Primes}\dfrac{1}{\beur(p)^\sigma}=+\infty
\end{equation*}
\end{proposition}

\begin{proof}

If $\sigma>\AFirst{\sigma}$ proposition \ref{stm::BeurExp} implies
\begin{eqnarray*}
	Z(\sigma)&=&\exp\left(
	\sum_{n=2}^{+\infty}\dfrac{\Lambda_\beur(n)}{\beur(n)^\sigma\log\beur(n)}
	\right)
	=\exp\left(
	\sum_{p\in\Primes}\sum_{m=1}^\infty\dfrac{1}{m\beur(p)^{m\sigma}}
	\right)\\
	&=&\exp\left(
	\sum_{p\in\Primes}\dfrac{1}{\beur(p)^{\sigma}}
	\right)\exp\bigl(h(\sigma)\bigr)\\
\end{eqnarray*}
where 
\begin{equation*}
	h(\sigma)=
		\sum_{p\in\Primes}\sum_{m=2}^\infty\dfrac{1}{m\beur(p)^{m\sigma}}
\end{equation*}
As in the classical case, using the fact that $\beur(n)\geq\beur_0>1$ when $n>1$
it is easy to show that the series $h(\sigma)$ converges for $\sigma>\AFirst{\sigma}/2$
and hence
\begin{equation*}
	\lim_{\sigma\to\AFirst{\sigma}^+}\sum_{p\in\Primes}\dfrac{1}{\beur(p)^\sigma}=
	\lim_{\sigma\to\AFirst{\sigma}^+}\bigl(\log Z(\sigma)
	-h(\sigma)\bigr)=+\infty.
\end{equation*}

\end{proof}

We now are ready to state and prove the main theorem of the paper.

\begin{theorem}\label{stm::ZetaBeurZeroFreeRegion}
Let $\APrev{\sigma}<\AFirst{\sigma}$ and let $\Domain\sset\CC$ be an open connected subset
containing the real half line $]\APrev{\sigma},+\infty[$
and contained in the open half space
$
\Re s>\APrev{\sigma}.\
$
Assume also that $\Domain$ is symmetric with respect to the real axis,
that is $s\in\Domain\implies\bar{s}\in\Domain$
and contains the open half space
$
\Re s>\AFirst{\sigma}.
$

Let
\begin{equation*}
\invZ:\Domain\to\CC
\end{equation*}
be a meromorphic function on $\Domain$
which is holomorphic on $\Domain\cap\RR\setminus\{\AFirst{\sigma}\}$
and 
\begin{equation*}
\invZ(s)=\sum_{n=1}^{\infty}\dfrac{a(n)}{\beur(n)^s}
\end{equation*}
when $\Re s>1$,
where $a(n)$ is a bounded completely multiplicative arithmetic function.

Assume also that 
\begin{equation*}
	Z(s)=\sum_{n=1}^{\infty}\dfrac{1}{\beur(n)^s}
\end{equation*}
extends to a meromorphic function on $\Domain$
with a unique simple pole at $s=\AFirst{\sigma}$.

Let $s_0\in\Domain\setminus\{\AFirst{\sigma}\}$ be given.
Then:

\begin{enumerate}
\item\label{stm::ZetaBeurMain::PingPong}
if $\AFirst{\sigma}$ is a zero or a pole of $L(s)$
and $Z(s_0)\neq0$
then
$s_0$ is a zero (resp. a pole) of $L(s)$
if, and only if,
$\bar{s_0}$ is a pole (resp. a zero) of $L(s)$;
\item\label{stm::ZetaBeur::CaseElleZeroCase}
if $L(\AFirst{\sigma})=0$
and $L(s)$ is holomorphic at $s=s_0$ and $s=\bar{s_0}$
then $Z(s_0)\neq0$;
in particular, if $L(s)$ is holomorphic on $\Domain$
then $\Domain$ is a zero free region for $Z(s)$;
\item\label{stm::ZetaBeur::ElleHalf}
if $L(\AFirst{\sigma})=0$ then
the boundary point $\APrev{\sigma}$ of $\Domain$ satisfies the constraint
\begin{equation*}
\APrev{\sigma}\geq\dfrac{\AFirst{\sigma}}{2}.
\end{equation*}
\end{enumerate}

\end{theorem}

\begin{proof}

Since $a(n)$ is completely multiplicative and bounded then necessarily
\begin{equation*}
\abs{a(n)}\leq1
\end{equation*}
for each $n\in\NN^+$
and by proposition \ref{stm::BeurExp} we also have
\begin{equation*}
\invZ(s)=\exp\left(
\sum_{n=1}^{\infty}\dfrac{a(n)\Lambda_\beur(n)}{\beur(n)^s\log\beur(n)}
\right)
\end{equation*}
and analogously
\begin{equation*}
Z(s)=\exp\left(
\sum_{n=1}^{\infty}\dfrac{\Lambda_\beur(n)}{\beur(n)^s\log\beur(n)}
\right).
\end{equation*}

Consider the functions
\begin{equation*}
F(s)=Z(s)^2\invZ(s)\bar{\invZ(\bar{s})}
\end{equation*}
and
\begin{equation*}
G(s)=\dfrac{Z(s)^2}{\invZ(s)\bar{\invZ(\bar{s})}}
\end{equation*}

When $\Re s>1$ 
\begin{equation*}
F(s)=\exp f(s),\ G(s)=\exp g(s)
\end{equation*}
where
\begin{equation*}
f(s)=\sum_{n=1}^{\infty}\dfrac{2\bigl(1+\Re a(n)\bigr)\Lambda_\beur(n)}{\beur(n)^s\log\beur(n)}.
\end{equation*}
and
\begin{equation*}
g(s)=\sum_{n=1}^{\infty}\dfrac{2\bigl(1-\Re a(n)\bigr)\Lambda_\beur(n)}{\beur(n)^s\log\beur(n)}.
\end{equation*}

Clearly $f(s)$ and $g(s)$ are Dirichlet series (absolutely) convergent when $\Re s>\AFirst{\sigma}$
with non negative coefficients
and hence their restriction to the real half line $]\AFirst{\sigma},+\infty[$
are completely monotone functions.

Assume now that $L(\AFirst{\sigma})=0$.
Theorem \ref{stm::cmHoloExtensionEx} then implies that
$F(s)$ never vanishes on $\Domain$ and hence
\begin{equation*}
	\invZ(s)\bar{\invZ(\bar{s})}=\dfrac{F(s)}{Z(s)^2}.
\end{equation*}
Similarly, if $L(s)$ has a pole at $s=\AFirst{\sigma}$ we obtain
\begin{equation*}
\invZ(s)\bar{\invZ(\bar{s})}=\dfrac{Z(s)^2}{G(s)}
\end{equation*}
with $G(s)$ never vanishing on $\Domain$.

In both cases \eqref{stm::ZetaBeurMain::PingPong} easily follows.

If $L(\AFirst{\sigma})=0$ then
$F(s)=Z(s)^2\invZ(s)\bar{\invZ(\bar{s})}$
never vanishes on $\Domain$ and 
\eqref{stm::ZetaBeur::CaseElleZeroCase} also follows.

It remains to prove \eqref{stm::ZetaBeur::ElleHalf},
that is
the inequality $\APrev{\sigma}\geq\AFirst{\sigma}/2$
assuming $L(\AFirst{\sigma})=0$.

For this purpose observe that theorem \ref{stm::cmHoloExtensionEx} also implies that
the defined above function
\begin{equation*}
f(s)=\sum_{n=1}^{\infty}\dfrac{2\bigl(1+\Re a(n)\bigr)\Lambda_\beur(n)}{\beur(n)^s\log\beur(n)}
\end{equation*}
extends holomorphically on the open half space
$\Re s>\APrev{\sigma}$.

The Landau's theorem \ref{stm::LandauStandard} also holds
for these generalized series
(see, e.g., \cite[Lema 15.1, pag. 463]{book:MontgomeryVaughan2006})
and hence the series $f(s)$ also
converges when $\Re s>\APrev{\sigma}$.

We will obtain the desired inequality $\APrev{\sigma}\geq\AFirst{\sigma}/2$
showing that the series $f(s)$ diverges at $s=\AFirst{\sigma}/2$.

An argument similar to the one given
at the end of the introduction of the paper yields
\begin{equation*}
	f(\sigma)\geq\dfrac{7}{8}\sum_{p\in\Primes}\dfrac{1}{\beur(p)^{2\sigma}}.
\end{equation*}
Proposition \ref{stm::beuPrimesDiverges} then implies
$f(\sigma)\to+\infty$ as $\sigma\to\AFirst{\sigma}/2$ from the left,
as desired.

\end{proof}

\section{\label{section:Corollaries}Corollaries}
In this section we will prove
several immediate corollaries of
theorem \ref{stm::ZetaBeurZeroFreeRegion}
which in a simple and unified manner
gives various non vanishing results
for zeta and $L$ like functions
on the boundary of the half plane
of absolute convergence.

Let
\begin{equation*}
	\beur:\NN^+\to\RR
\end{equation*}
be as in the previous section
with the associated ``zeta function''
\begin{equation*}
	Z(s)=\sum_{n=1}^{\infty}\dfrac{1}{\beur(n)^s}
\end{equation*}
having $\AFirst{\sigma}>0$ as abscissa of absolute convergence.

Let us begin with the ``prime number theorem''.
Such a theorem is already proved in the literature
(see \cite{article:MurtyOnSatoConj} and
\cite[Theorem 1.2, pag. 10]{book:MurtyMyrtyNonVanishingElleF})
but our proof is very easy.

\begin{corollary}\label{stm::ElleZetaPNT}
If $Z(s)$ extends in an open neighbourhood of
the closed half space
\begin{equation*}
	\Re s\geq\AFirst{\sigma}
\end{equation*}
to a meromorphic function with a unique simple pole at $s=\AFirst{\sigma}$
then
\begin{equation*}
	Z(\AFirst{\sigma}+it)\neq0
\end{equation*}
for each $t\in\RR\setminus\{0\}$.
\end{corollary}

\begin{proof}
Let $t\in\RR\setminus\{0\}$
and suppose that $Z(\AFirst{\sigma}+it)=0$.
Since $\bar{Z(\bar{s})}=Z(s)$ then also $Z(\AFirst{\sigma}-it)=0$.

Set $L(s)=Z(s+it)$.
When $\Re s>\AFirst{\sigma}$
\begin{equation*}
	L(s)=\sum_{n=1}^{\infty}\dfrac{\beur(n)^{-it}}{\beur(n)^s},
\end{equation*}
and the function $n\mapsto\beur(n)^{-it}$ is bounded and completely multiplicative.
We also have $L(\AFirst{\sigma})=0$ and $L(\AFirst{\sigma}-2it)=0$.
Then assertion \eqref{stm::ZetaBeurMain::PingPong}
of theorem \ref{stm::ZetaBeurZeroFreeRegion} implies that
$L(s)$ has a pole at $s=\AFirst{\sigma}+2it$,
that is $Z(s)$ has (another) pole at $s=\AFirst{\sigma}+3it$
and this contradicts the hypotheses made on $Z(s)$.

\end{proof}

Let also
\begin{equation*}
	a:\NN^+\to\RR
\end{equation*}
be a bounded completely multiplicative function
with the associated ``$L-$function''
\begin{equation*}
	\invZ(s)=\sum_{n=1}^{\infty}\dfrac{a(n)}{\beur(n)^s}
\end{equation*}

Most non vanishing theorem for $L-$ function associated
to various Dirichlet/Hecke characters follow
from the following statement.

\begin{corollary}\label{stm::ElleZetaZPole}
Let $Z(s)$ and $L(s)$ be given.
Assume that $L(s)$ is meromorphic on an open neighbourhood
of the closed half space
\begin{equation*}
\Re s\geq\dfrac{\AFirst{\sigma}}{2}
\end{equation*}
and also $Z(s)$ is holomorphic there
with the exception of a simple pole at
$s=\AFirst{\sigma}$.

If $L(\AFirst{\sigma}+it)=0$ for some $t\in\RR$
then the fuction $L(s)$ 
admit at least a pole
at $s=\sigma+it$ for some $\sigma$ satisfying
\begin{equation*}
	\dfrac{\AFirst{\sigma}}{2}\leq\sigma<\AFirst{\sigma}.
\end{equation*}

In particular, if $L(s)$ is holomorphic on such a neighbourhood
then
\begin{equation*}
L(\AFirst{\sigma}+it)\neq0
\end{equation*}
for each $t\in\RR$.

\end{corollary}

\begin{proof}
Let $t\in\RR$
and assume that $L(\AFirst{\sigma}+it)=0$.
Then the function
\begin{equation*}
L_t(s)\definedby L(s+it)
\end{equation*}
satisfies $L_t(\AFirst{\sigma})=0$.

If $L_t(s)$ has no poles at $s=\sigma\geq\AFirst{\sigma}/2$
then $L_t(s)$
is holomorphic up to the left of $s=\AFirst{\sigma}/2$,
contradicting
\eqref{stm::ZetaBeur::ElleHalf}
of theorem \ref{stm::ZetaBeurZeroFreeRegion}.

\end{proof}

Observe that if the Riemann hypothesis holds
the constraint $\Re s\geq\AFirst{\sigma}/2$ in the corollary above is
optimal.
Indeed, when $\AFirst{\sigma}=1$ and
$\beur(n)=n$, that is, $Z(s)=\zeta(s)$, the Riemann zeta function,
then the function
\begin{equation*}
L(s)=\dfrac{\zeta (2s)}{\zeta (s)}=\sum _{n=1}^{\infty }\dfrac{\lambda(n)}{n^s},
\end{equation*}
where $\lambda(n)$ is the Liouville function,
is holomorphic on the open half space $\Re s>1/2$ and $L(1)=0$.

We end
 
with a curiosity.

\begin{lemma}\label{stm::PingPongLemma}
Let $\Ping$ and $\Pong$ two subset of $\RR$.
Assume that $0\in\Ping$ and for each pair of
distincts reals $x,y$ whenever
\begin{equation*}
	\dfrac{x+y}{2}\in\Ping
\end{equation*}
then
\begin{equation*}
	x\in\Ping\Longleftrightarrow y\in\Pong.
\end{equation*}
If $\pong\in\Pong$ then $3\pong\in\Ping\cap\Pong$
and if $\ping\in\Ping$ then $-3\ping\in\Ping\cap\Pong$.
\end{lemma}

\begin{proof}
Since $0\in\Ping$ then for each $x\neq0$
\begin{equation*}
	x\in\Ping\Longleftrightarrow-x\in\Pong.
\end{equation*}
Let $\pong\in\Pong$ be given.
If $b=0$ then the assertion is trivially verified.
Assume hence that $\pong\neq0$.
 
Then we have $-\pong\in\Ping$
and
\begin{eqnarray*}
\dfrac{(-2\pong)+0}{2}=-\pong\in\Ping,\ 0\in\Ping\ &\implies&
	-2\pong\in\Pong,\ 2\pong\in\Ping,\\
\dfrac{(-3\pong)+\pong}{2}=-\pong\in\Ping,\ \pong\in\Pong\ &\implies&
	-3\pong\in\Ping,\ \boxed{3\pong\in\Pong},\\
\dfrac{\pong+3\pong}{2}=2\pong\in\Ping,\ \pong\in\Pong\ &
	\implies&\boxed{3\pong\in\Ping}.
\end{eqnarray*}
Thus we see that $3\pong\in\Ping\cap\Pong$, as required.

Let now $\ping\in\Ping$. 
If $a=0$ then the assertion is trivially verified
and
if $\ping\neq0$ then $-\ping\in\Pong$ and
we have just seen that then $-3\ping\in\Ping\cap\Pong$,
as desired.

\end{proof}

\begin{proposition}\label{stm::ElleZetaLNoPole}
Let $Z(s)$ and $L(s)$ be meromorphic
in an open neighbourhood of
the closed half space
\begin{equation*}
\Re s\geq\AFirst{\sigma}.
\end{equation*}
Assume that $Z(s)$ has a unique simple pole at $s=\AFirst{\sigma}$.
If $\AFirst{\sigma}$ is a pole or a zero of $L(s)$ then
$L(s)$ has no poles on the line $\Re s=\AFirst{\sigma}$
and $L(\AFirst{\sigma}+it)\neq0$ for each real $t\neq0$.
\end{proposition}

\begin{proof}
Assume first that $L(\AFirst{\sigma})=0$.
Let $\Ping, \Pong\sset\RR$  the set of $t\in\RR$ such that
$\AFirst{\sigma}+it$ is respectively a zero or a pole of $L(s)$.

We now show that $\Ping$ and $\Pong$ satisfy the hypotheses of
lemma \ref{stm::PingPongLemma}.

Of course $0\in\Ping$ being $L(\AFirst{\sigma})=0$.
Let $u,v\in\RR$ with $u\neq v$ and assume that
\begin{equation*}
	t\definedby\dfrac{u+v}{2}\in\Ping,
\end{equation*}
that is $L(\AFirst{\sigma}+it)=0$.

Consider then the function
\begin{equation*}
	L_t(s)\definedby L(s+it)
\end{equation*}
and set $s_0=\AFirst{\sigma}+i(u-t)$.

Since $L_t(\sigma)=0$ then
the assertion \eqref{stm::ZetaBeur::ElleHalf} of
theorem \ref{stm::ZetaBeurZeroFreeRegion} implies
that $s_0$ is a zero of $L_t(s)$
if, and only if,
$\bar{s_0}$ is a pole of $L_t(s)$.

Observing that and $\bar{s_0}=\AFirst{\sigma}+i(v-t)$
we obtain that $u\in\Ping$
if, and only if,
$v\in\Pong$,
as desired.

Since obviously $\Ping\cap\Pong=\void$
lemma \ref{stm::PingPongLemma} then forces
$\Pong=\void$ and $\Ping=\{0\}$
and this completes the proof of the proposition
in the case $L(\AFirst{\sigma})=0$.

Assume now that $\AFirst{\sigma}$ is a pole of $L(s)$.
Then it suffices to repeat the aforementioned argument interchanging
$\Ping$ with $\Pong$ and observing that the function $L_t(s)$
has a pole at $s=\AFirst{\sigma}$ instead of a zero.

The proof of the proposition is so completed.

\end{proof}

\bibliographystyle{amsalpha}

\begin{thebibliography}{New93}

\bibitem[Apo76]{book:ApostolANT}
Tom~M. Apostol, \emph{Introduction to analytic number theory}, Springer-Verlag,
  New York-Heidelberg, 1976, Undergraduate Texts in Mathematics. \MR{0434929}

\bibitem[Bat97]{article:BatemanOnInghamNoZeroes}
Paul~T. Bateman, \emph{A theorem of {I}ngham implying that {D}irichlet's
  {$L$}-functions have no zeros with real part one}, Enseign. Math. (2)
  \textbf{43} (1997), no.~3-4, 281--284. \MR{1489887}

\bibitem[Ber14]{article:Bernstein1914AM}
Serge Bernstein, \emph{Sur la d\'{e}finition et les propri\'{e}t\'{e}s des
  fonctions analytiques d'une variable r\'{e}elle}, Math. Ann. \textbf{75}
  (1914), no.~4, 449--468. \MR{1511806}

\bibitem[Beu37]{article:BeurlingPrimes}
Arne Beurling, \emph{Analyse de la loi asymptotique de la distribution des
  nombres premiers g\'{e}n\'{e}ralis\'{e}s. {I}}, Acta Math. \textbf{68}
  (1937), no.~1, 255--291. \MR{1577580}

\bibitem[Hil59]{book:HilleAFTVol1}
Einar Hille, \emph{Analytic function theory. {V}ol. 1}, Introduction to Higher
  Mathematics, Ginn and Company, Boston, 1959. \MR{0107692}

\bibitem[Ing30]{article:InghamNoteOnRiemannZeta}
A.~E. Ingham, \emph{Note on {R}iemann's zeta-{F}unction and {D}irichlet's
  {L}-{F}unctions}, J. London Math. Soc. \textbf{5} (1930), no.~2, 107--112.
  \MR{1574211}

\bibitem[Ing64]{book:InghamTheDistributionOfPrimeNumbers}
\bysame, \emph{The distribution of prime numbers}, Cambridge Tracts in
  Mathematics and Mathematical Physics, No. 30, Stechert-Hafner, Inc., New
  York, 1964. \MR{0184920}

\bibitem[Kno75]{book:KnopfmacherAbstractANT}
John Knopfmacher, \emph{Abstract analytic number theory}, North-Holland
  Publishing Co., Amsterdam-Oxford; American Elsevier Publishing Co., Inc., New
  York, 1975, North-Holland Mathematical Library, Vol. 12. \MR{0419383}

\bibitem[Lan01]{book:LangAlgebraicNumerTheory}
Serge Lang, \emph{{Algebraic Number Theory}}, Springer-Verlag, 2001.

\bibitem[MM97]{book:MurtyMyrtyNonVanishingElleF}
M.~Ram Murty and V.~Kumar Murty, \emph{Non-vanishing of {$L$}-functions and
  applications}, Progress in Mathematics, vol. 157, Birkh\"{a}user Verlag,
  Basel, 1997. \MR{1482805}

\bibitem[Mur82]{article:MurtyOnSatoConj}
V.~Kumar Murty, \emph{On the {S}ato-{T}ate conjecture}, Number theory related
  to {F}ermat's last theorem ({C}ambridge, {M}ass., 1981), Progr. Math.,
  vol.~26, Birkh\"{a}user, Boston, Mass., 1982, pp.~195--205. \MR{685296}

\bibitem[MV07]{book:MontgomeryVaughan2006}
Hugh~L. Montgomery and Robert~C. Vaughan, \emph{Multiplicative number theory.
  {I}. {C}lassical theory}, Cambridge Studies in Advanced Mathematics, vol.~97,
  Cambridge University Press, Cambridge, 2007. \MR{2378655}

\bibitem[Nar68]{article:NarasimhanRemaqueSurZeta}
Raghavan Narasimhan, \emph{Une remarque sur {$\zeta (1+it)$}}, Enseignement
  Math. (2) \textbf{14} (1968), 189--191 (1969). \MR{0249373}

\bibitem[Nar04]{book:NarkiewiczANT3}
W{\l}adys{\l}aw Narkiewicz, \emph{Elementary and analytic theory of algebraic
  numbers}, third ed., Springer Monographs in Mathematics, Springer-Verlag,
  Berlin, 2004. \MR{2078267}

\bibitem[Neu99]{book:NeukirchANT}
J\"{u}rgen Neukirch, \emph{Algebraic number theory}, Grundlehren der
  Mathematischen Wissenschaften [Fundamental Principles of Mathematical
  Sciences], vol. 322, Springer-Verlag, Berlin, 1999, Translated from the 1992
  German original and with a note by Norbert Schappacher, With a foreword by G.
  Harder. \MR{1697859}

\bibitem[New93]{article:NewmanNaturalProofElleNotZero}
D.~J. Newman, \emph{A ``natural'' proof of the nonvanishing of {$L$}-series}, A
  tribute to {E}mil {G}rosswald: number theory and related analysis, Contemp.
  Math., vol. 143, Amer. Math. Soc., Providence, RI, 1993, pp.~495--498.
  \MR{1210536}

\bibitem[Ogg70]{article:OggOnSatoTateConjecture}
A.~P. Ogg, \emph{A remark on the {S}ato-{T}ate conjecture}, Invent. Math.
  \textbf{9} (1969/1970), 198--200. \MR{0258835}

\bibitem[Ove14]{book:OverholtCourseAnalyticNT}
Marius Overholt, \emph{A course in analytic number theory}, Graduate Studies in
  Mathematics, vol. 160, American Mathematical Society, Providence, RI, 2014.
  \MR{3290245}

\bibitem[Pri94]{article:Pringsheim1894AnCont}
Alfred Pringsheim, \emph{Ueber {F}unctionen, welche in gewissen {P}unkten
  endliche {D}ifferentialquotienten jeder endlichen {O}rdnung, aber keine
  {T}aylor'sche {R}eihenentwickelung besitzen}, Math. Ann. \textbf{44} (1894),
  no.~1, 41--56. \MR{1510831}

\bibitem[Ser65]{article:SerreZandL1965}
Jean-Pierre Serre, \emph{Zeta and {$L$} functions}, Arithmetical {A}lgebraic
  {G}eometry ({P}roc. {C}onf. {P}urdue {U}niv., 1963), Harper \& Row, New York,
  1965, pp.~82--92. \MR{0194396}

\bibitem[SG60]{book:SansoneGerretsen}
Giovanni Sansone and Johan Gerretsen, \emph{Lectures on the theory of functions
  of a complex variable. {I}. {H}olomorphic functions}, P. Noordhoff,
  Groningen, 1960. \MR{0113988}

\bibitem[Sha49]{article:ShapiroOnNonVanishingOfElle}
George Shapiro, \emph{On the non-vanishing {$s=1$} of certain {D}irichlet
  series}, Amer. J. Math. \textbf{71} (1949), 621--626. \MR{0030552}

\bibitem[SSV12]{book:SchillingEtcBernsteinFunctions}
Ren\'{e}~L. Schilling, Renming Song, and Zoran Vondra\v{c}ek, \emph{Bernstein
  functions}, second ed., De Gruyter Studies in Mathematics, vol.~37, Walter de
  Gruyter \& Co., Berlin, 2012, Theory and applications. \MR{2978140}

\bibitem[Ten15]{book:TenenbaumAnalyticNumberTheory}
G{\'e}rard Tenenbaum, \emph{Introduction to analytic and probabilistic number
  theory}, third ed., Graduate Studies in Mathematics, vol. 163, American
  Mathematical Society, Providence, RI, 2015, Translated from the 2008 French
  edition by Patrick D. F. Ion. \MR{3363366}

\bibitem[Wid41]{book:WidderLaplace}
D.~V. Widder, \emph{{The Laplace Transform}}, Princeton University Press, 1941.

\bibitem[Win46]{article:WintnerFundLemmaDirichlet}
Aurel Wintner, \emph{The fundamental lemma in {D}irichlet's theory of the
  arithmetical progressions}, Amer. J. Math. \textbf{68} (1946), 285--292.
  \MR{0015422}

\end{thebibliography}

\providecommand{\bysame}{\leavevmode\hbox to3em{\hrulefill}\thinspace}
\providecommand{\MR}{\relax\ifhmode\unskip\space\fi MR }
\providecommand{\MRhref}[2]{
  \href{http://www.ams.org/mathscinet-getitem?mr=#1}{#2}
}
\providecommand{\href}[2]{#2}

\end{document}